\documentclass[12pt,oneside,reqno]{amsart}
\usepackage{mathrsfs, amsmath,amssymb}
\usepackage{fullpage}
\numberwithin{equation}{section}
\usepackage{graphicx}
\usepackage{svg}
\newtheorem{teo}{Theorem}[section]
\newtheorem{pro}[teo]{Proposition}
\newtheorem{lem}[teo]{Lemma}
\newtheorem{cor}[teo]{Corollary}
\newtheorem{rem}[teo]{Remark}

\def\D{\mathbb{D}}
\title{Liberation, free mutual information and orbital free entropy}
\subjclass[2010]{Primary 46L54; Secondary 94A17.}
\keywords{Liberation process; Herglotz transform;  L\"owner equations; Subordination; Free entropy.} 

\author[T. Hamdi]{Tarek Hamdi}
\address{Department of Management Information Systems \\ College of Business Administration, Qassim University \\ Saudi Arabia 
and Laboratoire d'Analyse Math\'ematiques et applications \\ LR11ES11 \\ Universit\'e de Tunis El-Manar \\ Tunisie
}
\email{tarek.hamdi@mail.com}

\begin{document}

\begin{abstract}
We present here a study of the liberation process for symmetries: $(R,S)\mapsto(R,U_tSU_t^*)$, where $U_t$ is a free unitary Brownian motion freely independent from $\{R,S\}$. More precisely, we use stochastic calculus to derive a partial differential equation (PDE for short) for the Herglotz transform of the process of unitary random variables $RU_tSU_t^*$ in the case of arbitrary trace values $\tau(R),\tau(S)$.   The obtained PDE is used to develop a theory of subordination in terms of L\"owner equations. 
On the other hand,  we present some connections between the liberation process for  symmetries  and its counterpart for projections when the symmetries and the projections are associated;
we relate the moments of their actions on the operators $X_t:=PU_tQU_t^*$ and $Y_t:=RU_tSU_t^*$
and use this to prove a relationship between the corresponding spectral measures (hereafter  $\mu_t$ and $\nu_t$).  
The paper is closed with an application of this study to the proof of the identity $i^*\left( \mathbb{C}P+\mathbb{C}(I-P); \mathbb{C}Q+\mathbb{C}(I-Q) \right)=-\chi_{orb}\left(P,Q\right)$.
\end{abstract}
\maketitle

\section{Introduction}
Let $(\mathscr{A},\tau)$ be a $W^*$-probability space and $U_t, t\in[0,\infty)$ a free unitary Brownian motion in $(\mathscr{A},\tau)$ with $U_0={\bf1}$.
For a given pair of orthogonal projections $\{P,Q\}$ in $\mathscr{A}$ that are freely independent from $(U_t)_{t\geq0}$, the so-called liberation process $(P,Q)\mapsto(P,U_tQU_t^*)$ was introduced in \cite{Voiculescu} in relation with the free entropy and the free Fisher information. 
 We look here to its counterpart $(R,S)\mapsto(R,U_tSU_t^*)$  when $\{R,S\}$ are two symmetries associated to $\{P,Q\}$ via  $R=2P-{\bf 1},S=2Q-{\bf 1}$. It is known, as consequence of the asymptotic freeness of $P$ and $U_tQU_t^*$, that the pair $(R,U_tSU_t^*)$ tends, as $t \rightarrow \infty$, to  $(R,USU^*)$ where $U$ is a Haar unitary free from $\{R,S\}$ and hence $R,USU^*$ are free (see \cite{Nic-Spe}). The connection between the two liberation processes can be understood by looking to the relationship between their actions on the operators $PU_tQU_t^*$ and $RU_tSU_t^*$. Thus, we mainly investigate this relationship in what follow.
The purpose of this study is to investigate the motivating question of proving $i^*=-\chi_{orb}$ for two  projections. An heuristic  argument for this question in \cite[Section 3.2]{Hia-Ued} supports that the equality holds.  Recently, Collins and Kemp \cite{Col-Kem} gave a proof of the equality for two projections with $\tau(P)=\tau(Q)=1/2$. This result was subsequently proved by Izumi and Ueda \cite{Izu-Ued}. They go further and use a subordination relation to give some partial results for the general case. 

In the present paper, we give an improved assertion of the result in \cite{Izu-Ued}  based on a similar subordination relation.  To this end, we study the dynamic of the unitary process  $Y_t=U_tRU_t^*S$. More precisely, we use stochastic calculus to derive a system of ODEs for its sequence of moments.  The obtained system is transformed into a PDE for  the Herglotz transform (hereafter $H(t,z)$) of its corresponding spectral measure $\nu_t$. In particular, we supply a full description of the measure of the steady-state solution. Then, we develop a theory of subordination for the process $Y_t$ akin to \cite{Izu-Ued} and obtain an explicit computation of the unique subordinate family. This allows us, in particular, to show that the boundary of its range is at a positive distance from $\pm1$ and use it to prove a certain regularity condition for the obtained  subordination relation. On the other hand, we generalize the approach used in \cite{Dem-Ham-Hmi} relating the moments of $X_t=PU_tQU_t^*$ and those of $Y_t=RU_tSU_t^*$ to the case of two arbitrary projections. The obtained relation is then transformed into a relationship between their corresponding measure $\mu_t$ and $\nu_t$.
  Finally, we  obtain a partial result for the identity $i^*=-\chi_{orb}$ in the case of arbitrary values of traces $\tau(P),\tau(Q)$ as application of the tools developed in this paper.

\section{Analysis of the spectral measure of $Y_t$}

\subsection{Sequence of moments}
Let $R,S\in \mathscr{A}$ be two symmetries with $\tau(R)=\alpha$ and $\tau(S)=\beta$ and $U_t, t\in[0,\infty)$ a free unitary Brownian motion freely independent from $\{R,S\}$.
Let $\nu_t$ be the spectral distribution  of the unitary process  $Y_t=RU_tSU_t^{*}$ on $\mathbb{T}$ (the set of complex numbers with modulus one). 
Our goal here is to derive a system of ODEs satisfied by the sequence of moments of $\nu_t$ via free stochastic calculus.
\begin{pro}\label{pro1} Let $f_n(t):=\tau\left[(RU_tSU_t^{*})^n \right] \quad n\geq1,t\geq0$,
then
\begin{equation*}
\partial_tf_1=-f_1+\alpha\beta ,
\end{equation*}
\begin{align*}
\partial_tf_n=-nf_n-n\displaystyle\sum_{k=1}^{n-1}f_{k}f_{n-k}+
\begin{cases}
\displaystyle n^2\alpha\beta   \quad {\rm if}\ n \ {\rm is\ odd}\\
\displaystyle n^2\frac{\alpha^2+\beta^2}{2}   \quad {\rm if}\ n \ {\rm is\ even}
\end{cases},
\quad n\geq2
\end{align*}
where $\alpha=\tau(R)$ and $\beta=\tau(S)$.
\end{pro}
\begin{proof}
Let $A_t=RU_tSU_t^{*}$, then using Ito's formula, we have
\begin{equation*}
d(A_t^n) =\sum_{k=1}^nA_t^{k-1}dA_tA_t^{n-k}+\sum_{1\leq j<k\leq n} A_t^{j-1}dA_tA_t^{k-j-1}dA_tA_t^{n-k}.
\end{equation*}
Taking the trace in both sides and use the trace property, we get
\begin{equation*}
\tau\left[d(A_t^n)\right] =\sum_{k=1}^n\tau\left[A_t^{n-1}dA_t\right]+\sum_{1\leq j<k\leq n} \tau\left[A_t^{n-(k-j)-1}dA_tA_t^{k-j-1}dA_t\right].
\end{equation*}
The first summands do not depend on the summation variable $k$, while the second summands depend on the summation variable $j,k$ only through their difference $k-j$. Then re-indexing by  $l=k-j$, we get
\begin{equation*}
\tau\left[d(A_t^n)\right] =n\tau\left[A_t^{n-1}dA_t\right]+\sum_{l=1}^{n-1}\sum_{1\leq j<k\leq n, \ k-j=l} \tau\left[A_t^{n-l-1}dA_tA_t^{l-1}dA_t\right].
\end{equation*}
Since the number of pairs $(j,k)$ such that $k-j=l$ for fixed $l$ is equal to $n-l$, then the second summation becomes
\begin{equation}\label{sum1}
\sum_{l=1}^{n-1}(n-l) \tau\left[A_t^{n-l-1}dA_tA_t^{l-1}dA_t\right].
\end{equation}
This sum rewrites, after re-indexing $k=n-l$, as
\begin{equation}\label{sum2}
\sum_{k=1}^{n-1}k \tau\left[A_t^{k-1}dA_tA_t^{n-k-1}dA_t\right].
\end{equation}
Using the trace property and adding the summations (\ref{sum1}) and (\ref{sum2}), we get
\begin{equation*}
\sum_{k=1}^{n-1}(n-k+k) \tau\left[A_t^{n-k-1}dA_tA_t^{k-1}dA_t\right]=n\sum_{k=1}^{n-1} \tau\left[A_t^{n-k-1}dA_tA_t^{k-1}dA_t\right].
\end{equation*}
Thus, we have
\begin{equation}\label{Equ1}
\tau\left[d(A_t^n)\right] =n\tau\left[A_t^{n-1}dA_t\right]+\frac{n}{2}\sum_{k=1}^{n-1} \tau\left[A_t^{n-k-1}dA_tA_t^{k-1}dA_t\right].
\end{equation}
Now since $R$ and $S$ are independent from $t$, the free Ito's formula implies 
\begin{align*}
dA_t=Rd(R_tSU_t^*)&=R(dU_t)SU_t^*+RU_td(SU_t^*)+R(dU_t)d(SU_t^*)
\\&=R(dU_t)SU_t^*+RU_tS(dU_t^*)+R(dU_t)S(dU_t^*).
\end{align*}
But, since
\begin{equation*}
dU_t=iU_tdB_t-\frac{1}{2}U_t dt \quad {\rm and}\ dU_t^*=-idB_tU_t^*-\frac{1}{2}U_t^* dt.
\end{equation*}
Then substituting these equations in the expression of $dA_t$ we get
\begin{align*}
dA_t&=R(iU_tdB_t-\frac{1}{2}U_t dt)SU_t^*+RU_tS(-idB_tU_t^*-\frac{1}{2}U_t^* dt)+R(iU_tdB_t-\frac{1}{2}U_t dt)S(-idB_tU_t^*-\frac{1}{2}U_t^* dt).
\end{align*}
The first two terms simplify to
\begin{align*}
iRU_tdB_t SU_t^*-iRU_tSdB_tU_t^*-RU_tSU_t^*dt=iRU_tdB_t SU_t^*-iRU_tSdB_tU_t^*-A_tdt
\end{align*}
while the last term is reduced to
\begin{align*}
R(iU_tdB_t)S(-idB_tU_t^*)=RU_tdB_tSdB_tU_t^*=RU_t\tau(S)U_t^*dt=\beta R dt
\end{align*}
Thus, we have
\begin{equation}\label{sde1}
dA_t=iRU_tdB_tSU_t^*-iRU_tSdB_tU_t^*+(\beta R-A_t)dt.
\end{equation}
So that,
\begin{equation*}
A_t^{n-1}dA_t=iA_t^{n-1}RU_tdB_tSU_t^*-iA_t^{n-1}RU_tSdB_tU_t^*+A_t^{n-1}(\beta R-A_t)dt.
\end{equation*}
Since the trace of a stochastic integral is zero, then the first term in equation (\ref{Equ1}) is given by
\begin{equation*}
\tau(A_t^{n-1}dA_t)=\tau\left[A_t^{n-1}(\beta R-A_t)\right]dt=\left[\beta \tau(A_t^{n-1}R)-\tau(A_t^n)\right]dt.
\end{equation*}
Using the trace property and the relations $R^2=S^2=U_tU_t^*=1$, we have $ \tau(A_t^{n-1}R)=\tau(R)=\alpha$ if $n$ is odd and  $\tau(A_t^{n-1}R)=\tau(S)=\beta$  otherwise.

Hence, the first term in equation (\ref{Equ1}) is equal to
\begin{equation}\label{summation1}
 n\tau(A_t^{n-1}dA_t)=\begin{cases}
 \left[n\beta^2-n\tau(A_t^n)\right]dt\quad {\rm if}\ n \ {\rm is\ even}\\
 \left[n\beta \alpha-n\tau(A_t^n)\right]dt\quad {\rm otherwise}
 \end{cases}.
\end{equation}
For the second term in equation (\ref{Equ1}), we shall use the following result.
\begin{lem}\label{lem1}
Let 
\begin{equation}\label{sde2}
dZ_t=iRU_tdB_tSU_t^*-iRU_tSdB_tU_t^*.
\end{equation}
Then
\begin{equation*}
dtdZ_t=dZ_tdt=(dt)^2=0
\end{equation*}
and for any adapted process $V_t$, we have
\begin{equation}\label{equ2}
dZ_tV_tdZ_t=[2R\tau(RV_t)-2A_t\tau(A_tV_t)]dt.
\end{equation}
\end{lem}
\begin{proof}
The first statement is a  consequence of It\^o rules since $Z_t$ is a stochastic integral. For the last, we expand
\begin{align*}
dZ_tV_tdZ_t=&(iRU_tdB_tSU_t^*-iRU_tSdB_tU_t^*)V_t(iRU_tdB_tSU_t^*-iRU_tSdB_tU_t^*)
\\=&-RU_tdB_tSU_t^*V_tRU_tdB_tSU_t^* +RU_tdB_tSU_t^*V_tRU_tSdB_tU_t^* +RU_tSdB_tU_t^*V_tRU_tdB_tSU_t^* 
\\&-RU_tSdB_tU_t^*V_tRU_tSdB_tU_t^*.
\end{align*}
Applying the It\^o rule
\begin{equation*}
dB_tV_tdB_t=\tau(V_t)dt
\end{equation*}
to each of these terms yields
\begin{align*}
dZ_tV_tdZ_t=&-RU_t\tau(SU_t^*V_tRU_t)SU_t^* dt+RU_t\tau(SU_t^*V_tRU_tS)U_t^*dt +RU_tS\tau(U_t^*V_tRU_t)SU_t^* dt
\\&-RU_tS\tau(U_t^*V_tRU_tS)U_t^*dt.
\end{align*}
Using the trace property and the relations $S^2=U_tU_t^*=1,A_t=RU_tSU_t^*$, we get
\begin{align*}
dZ_tV_tdZ_t=&-A_t\tau(A_t^*V_t) dt+R\tau(V_tR)dt +R\tau(V_tR) dt-A_t\tau(A_t^*V_t)dt
\end{align*}
which simplifies to give the equality (\ref{equ2}).
\end{proof}
It follows from (\ref{sde1}) and (\ref{sde2}) that for $n\geq2$ and $k\in\{1,\ldots,n-1\}$, 
\begin{equation*}
A_t^{n-k-1}dA_tA_t^{k-1}dA_t=A_t^{n-k-1}[dZ_t+(\beta R-A_t)dt]A_t^{k-1}[dZ_t+(\beta R-A_t)dt]
\end{equation*}
which expands into four terms. But by use of lemma \ref{lem1}, the only surviving term is
\begin{equation*}
A_t^{n-k-1}dZ_tA_t^{k-1}dZ_t=A_t^{n-k-1}[2R\tau(RA_t^{k-1})-2A_t\tau(A_t^{k})]dt.
\end{equation*}
Taking the trace, we get
\begin{equation*}
\tau(A_t^{n-k-1}dA_tA_t^{k-1}dA_t)=[2\tau(RA_t^{k-1})\tau(RA_t^{n-k-1})-2\tau(A_t^{k})\tau(A_t^{n-k})]dt
\end{equation*}
Using the same consideration leading to (\ref{sum1}) and the fact that if $n$ is even then $k,n-k$ have the same parity and if $n$ is odd then $k,n-k$ have opposite parity, we have
\begin{equation*}
\tau(A_t^{n-k-1}dA_tA_t^{k-1}dA_t)=
\begin{cases}
(2\alpha^2-2\tau(A_t^{k})\tau(A_t^{n-k}))dt\quad {\rm if}\ n \ {\rm is\ even\ and\ } k\ {\rm is\ odd}\\
(2\beta^2-2\tau(A_t^{k})\tau(A_t^{n-k}))dt\quad {\rm if}\ n \ {\rm is\ even\ and\ } k\ {\rm is\  even}\\
(2\alpha\beta-2\tau(A_t^{k})\tau(A_t^{n-k}))dt\quad {\rm if}\ n \ {\rm is\ odd\ and\ } k\ {\rm is\ odd}\\
(2\alpha\beta-2\tau(A_t^{k})\tau(A_t^{n-k}))dt\quad {\rm if}\ n \ {\rm is\ odd\ and\ } k\ {\rm is\ even}
\end{cases}
\end{equation*}
Hence, the second term in equation (\ref{Equ1}) is equal to
\begin{align*}\displaystyle
\frac{n}{2}\displaystyle\sum_{k=1}^{n-1}\tau(A_t^{n-k-1}dA_tA_t^{k-1}dA_t)&=
\begin{cases}
\left(-n\displaystyle\sum_{k=1}^{n-1}\tau(A_t^{k})\tau(A_t^{n-k})+\frac{n^2}{2}\alpha^2+\frac{n(n-2)}{2}\beta^2\right)dt\quad {\rm if}\ n \ {\rm is\ even}\\
\left(-n\displaystyle\sum_{k=1}^{n-1}\tau(A_t^{k})\tau(A_t^{n-k})+n(n-1)\alpha\beta\right)dt\quad {\rm if}\ n \ {\rm is\ odd}
\end{cases}
\end{align*}
which simplifies to
\begin{align}\label{summation2}
\frac{n}{2}\displaystyle\sum_{k=1}^{n-1}\tau(A_t^{n-k-1}dA_tA_t^{k-1}dA_t)&=-n\displaystyle\sum_{k=1}^{n-1}\tau(A_t^{k})\tau(A_t^{n-k})+
\begin{cases}
\left(\frac{n^2}{2}\alpha^2+\frac{n(n-2)}{2}\beta^2\right)dt\quad {\rm if}\ n \ {\rm is\ even}\\
\left(n(n-1)\alpha\beta\right)dt\quad {\rm if}\ n \ {\rm is\ odd}
\end{cases}
\end{align}
and hence the desired assertions follows after summing (\ref{summation1}) and (\ref{summation2}).
\end{proof}

\subsection{The Herglotz transform of $\nu_t$}
Here, we derive a PDE governing the Herglotz transform of the spectral measure $\nu_t$:
\begin{equation*}
 H(t,z):=\int_{\mathbb{T}}\frac{\zeta+z}{\zeta-z}d\nu_t(\zeta)=1+2\sum_{n\geq 1} f_n(t) z^n.
\end{equation*}
 Recall that, this is an analytic function on $\mathbb{D}$ (the open unit disc of $\mathbb{C}$). 
\begin{pro}
The function $H(t,z)$ satisfies the PDE  
\begin{equation}\label{pde}
\partial_tH+\frac{z}{2}\partial_zH^2= \frac{2 z \left(\alpha z^2+2 \beta z+\alpha\right) \left(\beta z^2+2  \alpha z+\beta\right)}{\left(1-z^2\right)^3}.
\end{equation}
\end{pro}
\begin{proof}
By direct calculation from Proposition \ref{pro1}, we have
\begin{align*}
\partial_tH&=2\sum_{n\geq 1} \partial_tf_n(t) z^n
\\&=-2\sum_{n\geq 1}nf_nz^n-2\sum_{n\geq 1}n\sum_{k=1}^{n-1}f_{k}f_{n-k}z^n+(\alpha^2+\beta^2) \sum_{n\geq 1,\ n\ {\rm even}}n^2z^n+2\alpha\beta\sum_{n\geq 1,\ n\ {\rm odd}}n^2z^n
\\&=-z\partial_zH-2\sum_{k\geq 1}f_kz^k\sum_{n\geq k+1}nf_{n-k}z^{n-k}+4(\alpha^2+\beta^2)\frac{z^2(1+z^2)}{(1-z^2)^3}+2\alpha\beta z\frac{1+6z^2+z^4}{(1-z^2)^3}
\\&=-z\partial_zH-4\frac{H-1}{2}\sum_{n\geq k+1}nf_{n-k}z^{n-k}+\frac{2 z \left(\alpha z^2+2 \beta z+\alpha\right) \left(\beta z^2+2 \alpha z+\beta\right)}{\left(1-z^2\right)^3}
\\&=-zH\partial_zH+\frac{2 z \left( \alpha z^2+2 \beta z+\alpha\right) \left(\beta z^2+2 \alpha z+\beta\right)}{\left(1-z^2\right)^3}.
\end{align*}
\end{proof}

\subsection{Steady-state solution}
As mentioned in the Introduction, it is known from the asymptotic freeness of $P$ and $U_tQU_t^*$ that 
\begin{pro}
The spectral measure $\nu_t$ of $RU_tSU_t^*$ converges weakly, as $t\rightarrow\infty$, to the free multiplicative convolution of the spectral measures of $R$ and $USU^*$, where $U\in\mathscr{A}$ is a Haar unitary operator free from $\{R,S\}$.
\end{pro}
We will see this directly from the PDE \eqref{pde}.
Let $H(\infty,.)$ be the state solution of (\ref{pde}), then it satisfies
\begin{equation*}
\partial_zH^2= \frac{4 \left(\alpha z^2+2 \beta z+\alpha\right) \left(\beta z^2+2  \alpha z+\beta\right)}{\left(1-z^2\right)^3}.
\end{equation*}
After integration and taking into account $H(\infty,0)=1$, we get
\begin{equation}\label{stationary}
H(\infty,z)=\sqrt{1+4 z\frac{\alpha \beta \left(1+z\right)^2+ \left(\alpha-\beta \right)^2z}{\left(1-z^2\right)^2}}
\end{equation}
where the principal branch of the square root is taken. On the other hand, the next technical proposition gives an explicit calculation for  the Herglotz transform of $\nu_R\boxtimes \nu_S$.
\begin{pro}\label{freeconv}
Let $\mu=\frac{1+\alpha}{2}\delta_1+\frac{1-\alpha}{2}\delta_{-1}$ and 
\begin{align*}
\nu=\left(\frac{1+\alpha}{2}\delta_1+\frac{1-\alpha}{2}\delta_{-1} \right) \boxtimes \left(\frac{1+\beta}{2}\delta_1+\frac{1-\beta}{2}\delta_{-1} \right)
\end{align*}
for $\alpha,\beta\in (-1,1]$. Then the Herglotz transform of $ \nu$ is given by
\begin{align*}
H_{ \nu}(z)=H(\infty,z)=\sqrt{1+4 z\frac{\alpha \beta \left(1+z\right)^2+ \left(\alpha-\beta \right)^2z}{\left(1-z^2\right)^2}}.
\end{align*}
\end{pro}
\begin{proof}
Using the analytic machinery for multiplicative convolution (see \cite{Dyk-Nic-Voi}), we have
\begin{equation*}
\psi_{\mu}(z)=\frac{z(z+\alpha)}{1-z^2},
\end{equation*}
\begin{equation*}
\chi_{\mu}(z)=\frac{-\alpha\pm\sqrt{\alpha^2+4z(z+1)}}{2(z+1)},
\end{equation*}
\begin{equation*}
S_{\mu}(z)=\frac{-\alpha\pm\sqrt{\alpha^2+4z(z+1)}}{2z}.
\end{equation*}
So that
\begin{equation*}
S_{ \nu}(z)=\frac{\left(-\alpha\pm\sqrt{\alpha^2+4z(z+1)}\right)\left(-\beta\pm\sqrt{\beta^2+4z(z+1)}\right)}{4z^2},
\end{equation*}
\begin{equation*}
\chi_{ \nu}(z)=\frac{\left(-\alpha\pm\sqrt{\alpha^2+4z(z+1)}\right)\left(-\beta\pm\sqrt{\beta^2+4z(z+1)}\right)}{4z(z+1)},
\end{equation*}
and $\psi_{ \nu}$ satisfies
\begin{equation*}
\frac{\left(-\alpha\pm\sqrt{\alpha^2+4\psi_{ \nu}(\psi_{ \nu}+1)}\right)\left(-\beta\pm\sqrt{\beta^2+4\psi_{ \nu}(\psi_{ \nu}+1)}\right)}{4\psi_{ \nu}(\psi_{ \nu}+1)}=z.
\end{equation*}
Letting $\varphi_{ \nu}=\psi_{ \nu}(\psi_{ \nu}+1)$, we get $\psi_{ \nu}=(-1\pm\sqrt{1+4\varphi_{ \nu}})/2$ and since the Herglotz transform has a positive real part,   $H _{ \nu}=\sqrt{1+4\varphi_{ \nu}}$ where $\varphi_{ \nu}$ is given by
\begin{equation*}
\frac{\left(-\alpha\pm\sqrt{\alpha^2+4\varphi_{ \nu}}\right)\left(-\beta\pm\sqrt{\beta^2+4\varphi_{ \nu}}\right)}{4\varphi_{ \nu}}=z.
\end{equation*}
Or equivalently
\begin{equation*}
-\alpha\pm\sqrt{\alpha^2+4\varphi_{ \nu}}=z\left(\beta\pm\sqrt{\beta^2+4\varphi_{ \nu}} \right).
\end{equation*}
Rearranging this last equality and raising it to the square, we get
\begin{equation*}
\alpha^2+4\varphi_{ \nu}+z^2(\beta^2+4\varphi_{ \nu})-(\alpha+ \beta z)^2=2z\sqrt{(\alpha^2+4\varphi_{ \nu})(\beta^2+4\varphi_{ \nu})}.
\end{equation*}
So we raise it to the square once again, to get
\begin{equation*}
\left[\alpha^2+4\varphi_{ \nu}+z^2(\beta^2+4\varphi_{ \nu})-(\alpha+ \beta z)^2\right]^2=4z^2(\alpha^2+4\varphi_{ \nu})(\beta^2+4\varphi_{ \nu}).
\end{equation*}
Which simplifies to
\begin{equation*}
2(1-z^2)^2\varphi_{ \nu}+[(1-z^2)\left(\alpha^2-\beta^2z^2-(\alpha+\beta z)^2  \right)-2z^2(\alpha+\beta z)^2]=0.
\end{equation*}
Finally, 
\begin{align*}
\varphi_{ \nu}(z)=\frac{\alpha \beta z\left(1+z\right)^2+ \left(\alpha-\beta \right)^2z^2}{\left(1-z^2\right)^2}
\end{align*}
as desired.

\end{proof}
The next proposition provides a Lebesgue decomposition of the spectral measure $\nu_\infty$.
\begin{pro}
One has
\begin{equation*}
\nu_\infty=a\delta_{\pi}+b\delta_0+\frac{\sqrt{-(\cos\theta-r_+)(\cos\theta-r_-)}}{2\pi|\sin\theta|}{\bf 1}_{(\theta_-,\theta_+)\cup(-\theta_+,-\theta_-)}d\theta
\end{equation*}
with
\begin{equation*}
a=\frac{|\alpha- \beta|}{2}, b=\frac{|\alpha+ \beta|}{2}, r_{\pm}=\alpha\beta\pm\sqrt{(1-\alpha^2)(1-\beta^2)}\quad {\rm and}\ \theta_{\pm}=\arccos r_{\pm}.
\end{equation*}
\end{pro}
\begin{proof}
Writing \eqref{stationary} as
\begin{equation*}
H(\infty,z)=\frac{\sqrt{\left(1-z^2\right)^2+4 z[\alpha \beta \left(1+z\right)^2+ \left(\alpha-\beta \right)^2z]}}{\left(1-z^2\right)},
\end{equation*}
it follows that $H(\infty,.)$ admits two simple poles at $z=1$ and $z=-1$. So that, the decomposition of $\nu_\infty$ is given by
\begin{equation*}
\nu_\infty=a\delta_{\pi}+b\delta_{0}+\Re\left[H(\infty,e^{i\theta}) \right]\frac{d\theta}{2\pi}
\end{equation*}
where $d\theta$ denotes  the (no-normalized) Lebesgue measure on $\mathbb{T}=(-\pi,\pi]$ and $a,b$ are the residue of $\frac{1}{2}H(\infty,.)$ at $-1,1$. Thus, we have
\begin{equation*}
a=\lim_{z\rightarrow-1}\frac{\sqrt{\left(1-z^2\right)^2+4 z[\alpha \beta \left(1+z\right)^2+ \left(\alpha-\beta \right)^2z]}}{2\left(1+z\right)}=\frac{|\alpha- \beta|}{2},
\end{equation*}
\begin{equation*}
b=\lim_{z\rightarrow1}\frac{\sqrt{\left(1-z^2\right)^2+4 z[\alpha \beta \left(1+z\right)^2+ \left(\alpha-\beta \right)^2z]}}{2\left(1-z\right)}=\frac{|\alpha+ \beta|}{2}
\end{equation*}
and the density is given by direct calculation

\begin{align*}
\Re\left[H(\infty,e^{i\theta}) \right]&=\Re\left[\sqrt{1+4 e^{i\theta}\frac{\alpha \beta \left(1+e^{i\theta}\right)^2+ \left(\alpha-\beta \right)^2e^{i\theta}}{\left(1-e^{2i\theta}\right)^2}}\right]
\\&=\Re\left[\sqrt{1+\frac{4\alpha \beta e^{i\theta}}{(1-e^{i\theta})^2}+\frac{4  \left(\alpha-\beta \right)^2e^{2i\theta}}{\left(1-e^{2i\theta}\right)^2}}\right]
\\&=\sqrt{1-\frac{\alpha \beta }{\sin^2\frac{\theta}{2}}-\frac{  \left(\alpha-\beta \right)^2}{\sin^2\theta}}
\\&=\frac{\sqrt{\sin^2\theta-4\alpha\beta\cos^2\frac{\theta}{2}-(\alpha-\beta )^2}}{|\sin\theta|},
\end{align*}
where we have used in the last equality the relation
\begin{equation*}
\sin^2\frac{\theta}{2}=\frac{\sin^2\theta}{4\cos^2\frac{\theta}{2}}.
\end{equation*}
Finally, by use of the basic trigonometric identities:
\begin{equation*}
\cos^2\theta+\sin^2\theta=1\quad {\rm and}\ \cos^2\frac{\theta}{2}=\frac{1+\cos\theta}{2},
\end{equation*}
the denominator rewrites as
\begin{align*}
\sin^2\theta-4\alpha\beta\cos^2\frac{\theta}{2}-(\alpha-\beta )^2&=1-\cos^2\theta-2\alpha\beta\cos\theta-2\alpha\beta-(\alpha-\beta )^2
\\&=-\cos^2\theta-2\alpha\beta\cos \theta+1-\alpha^2-\beta^2.
\end{align*}
Using the discriminant $\Delta=4(\alpha^2\beta^2+1-\alpha^2-\beta^2)=4(1-\alpha^2)(1-\beta^2)\geq0$, we get the factorization
$-(\cos\theta-r_+)(\cos\theta-r_-)$ with
\begin{equation*}
r_{\pm}=\alpha\beta\pm\sqrt{(1-\alpha^2)(1-\beta^2)}.
\end{equation*}
\end{proof}

\begin{rem}
It should be noted that this measure appears in \cite[Example 4.5]{Hia-Pet} as the distribution of $e^{i\pi P}e^{-i\pi Q}$ for a pair of free projections $\{P,Q\}$ in $\mathscr{A}$. In particular, when $\alpha=\beta=0$ (i.e. $\tau(P)=\tau(Q)=1/2$), it coincides with the uniform measure on $\mathbb{T}$.
\end{rem}

\section{Subordination for the liberation of symmetries}
The aim of this section is to derive a subordination results in terms of L\"owner equations and give an explicit formula for the unique subordinate family. 
\begin{pro}
Let $H$ be a solution to the PDE (\ref{pde}). Then there exists a unique subordinate family of conformal self-maps $\phi_t$ on $\mathbb{D}$ such that
\begin{equation}\label{flot}
H(t,\phi_t(z))^2-H(\infty,\phi_t(z))^2=H(0,z)^2-H(\infty,z)^2.
\end{equation}
\end{pro}
\begin{proof}
Differentiating the  characteristic curve $t\mapsto (\phi_t(z),H(t,\phi_t(z)))$ associated with the PDE (\ref{pde}), we get the following system of ODEs:
\begin{equation}\label{Low1}
\partial_t\phi_{t} = \phi_{t} H(t,\phi_{t}), \quad \phi_{0}(z) = z,
\end{equation}
\begin{equation}\label{Low2}
\partial_t \left[H(t,\phi_{t})\right] =\frac{4(\alpha^2+\beta^2)\phi_t^2(1+\phi_t^2)+2\alpha\beta \phi_t(1+6\phi_t^2+\phi_t^4)}{(1-\phi_t^2)^{3}}.
\end{equation}
The ODE \eqref{Low1} is the radial L\"owner equation driven by the Herglotz function $H$. Then $\phi_{t}$ is a conformal map from $\Omega_t:=\{z, T_{z} > t\}$ onto $\mathbb{D}$  (see, e.g., Theorem 4.14 in \cite{Law}), where $T_{z}$ is the supremum of all $t$ such that $\phi_{t}(z) \in \mathbb{D}$ for fixed $z \in \mathbb{D}$.
The ODE \eqref{Low2}, combined with \eqref{Low1}, shows that
\begin{equation}
H\partial_tH=\frac{4(\alpha^2+\beta^2)\phi_t(1+\phi_t^2)+2\alpha\beta (1+6\phi_t^2+\phi_t^4)}{(1-\phi_t^2)^{3}}\partial_t\phi_t.
\end{equation}
Which implies, after integrating with respect to $t$, that
\begin{align*}
H(t,\phi_t(z))^2-H(0,z)^2=&4\frac{(\alpha^2+\beta^2)\phi_t(z)^2+2\alpha\beta \phi_t(z)(1+\phi_t(z)^2)}{(1-\phi_t(z)^2)^{2}}
\\&-4\frac{(\alpha^2+\beta^2)z^2+\alpha\beta z(1+z^2)}{(1-z^2)^{2}}.
\end{align*}
This proves the proposition. 
\end{proof}
\begin{rem}\label{remark}
When $P,Q$ are two projections associated to $R,S$ such that $\tau(P)=\tau(Q)=1/2$ (i.e. $\alpha=\beta=0$), the function $t\mapsto H(t,\phi_t(z))$ is constant, so that $H(t,\phi_t(z))=H(0,z)$. 
Then, $\phi_t(z)=ze^{tH(0,z)}$. This enables us to retrieve the description of $\nu_{t/2}$ in \cite[Proposition 3.3]{Izu-Ued}.
In particular, when $P=Q$ and  $\nu_0=\delta_0$ (i.e. $H(0,z)=(1+z)/(1-z)$), we retrieve the description in \cite[Corollary 3.3]{Dem-Ham-Hmi} of the spectral measure $\mu_t$ on $[0,1]$ of the free Jacobi process (the process $X_t$ viewed as a random variable in the compressed probability space $(P\mathcal{A}P,\frac{1}{\tau(P)}\tau)$). 
\end{rem}

For any $t\geq0$, define\footnote{We take the principal branch of the square root.}
\begin{equation}\label{defK}
K(t,z):=\sqrt{H(t,z)^2-\left(a\frac{1-z}{1+z}+b\frac{1+z}{1-z} \right)^2},\quad |z|<1.
\end{equation}
This function is  analytic in $\mathbb{D}$ with positive real part. Indeed, the function
\begin{equation*}
H(t,z)^2-\left(a\frac{1-z}{1+z}+b\frac{1+z}{1-z} \right)^2,\quad |z|<1
\end{equation*}
can not take negative value in $\D$ since  the two measures $\nu_t-a\delta_{\pi}-b\delta_0$ and $\nu_t+a\delta_{\pi}+b\delta_0$ are finite positive measure in $\mathbb{T}$ (see Proposition \ref{atom} below).
Thus, according to the Herglotz theorem (see \cite[Theorem 1.8.9]{CMR}), there exists a unique probability measure $\gamma_t$ in $\mathbb{T}$ such that 
\begin{align}\label{gamma}
K(t,z)=\int_{\mathbb{T}}\frac{\zeta+z}{\zeta-z}d\gamma_t(\zeta).
\end{align}
\begin{rem}\label{constant}
By \eqref{pde}, the function $K(t,z)$ satisfies 
\begin{equation*}
\partial_tK+zH(t,z)\partial_zK=0
\end{equation*}
 and, in the time stationary case, $K(\infty,z)$ becomes the constant $\sqrt{1-\max\{\alpha^2,\beta^2\}}$ thanks to \eqref{defK} together with \eqref{stationary}.
\end{rem}
 Let $\eta_t$ be the inverse of $\phi_t:\Omega_t\rightarrow\D$. It is known (see, e.g., \cite[Remark 4.15]{Law}) that $\eta_t$ satisfies 
\begin{equation*}
\partial_t\eta_t(z)=-z\partial_z\eta_t(z)H(t,z), \quad \eta_0(z)=z,
\end{equation*}
the radial L\"owner PDE driven by the probability measure $\nu_t$. 
Here is an exact subordination relation.
\begin{pro} \label{subor}
The equality  $K(t,z)=K(0,\eta_t(z))$ holds for any $z\in\mathbb{D}$ and $t\geq0$.
\end{pro}
\begin{proof}
From \eqref{flot}, we have
\begin{align*}
K(t,\phi_t(z))^2= H(0,z)^2-H(\infty,z)^2+H(\infty,\phi_t(z))^2-\left(a\frac{1-\phi_t(z)}{1+\phi_t(z)}+b\frac{1+\phi_t(z)}{1-\phi_t(z)} \right)^2.
\end{align*}
But
\begin{align*}
H(\infty,z)^2&=1+4 z\frac{\alpha \beta \left(1+z\right)^2+ \left(\alpha-\beta \right)^2z}{\left(1-z^2\right)^2}
\\&=1-\max\{\alpha^2,\beta^2\}+\left(a\frac{1-z}{1+z}+b\frac{1+z}{1-z} \right)^2.
\end{align*}
Then
\begin{align*}
K(t,\phi_t(z))^2= H(0,z)^2-\left(a\frac{1-z}{1+z}+b\frac{1+z}{1-z} \right)^2,
\end{align*}
and we are done.
\end{proof}

The next proposition gives an explicit expression for the subordinate family $(\phi_t)_{t\geq0}$. 

\begin{pro} \label{expression} For any $t\geq0$ and $z\in\Omega_t\cap\mathbb{R}$, we have 
\begin{align*}
\phi_t(z)=\frac{w_t(y)-1}{w_t(y)+1},
\end{align*}
with
\begin{align*}
w_t(y)=\sqrt{\frac{\left(b^2-a^2-c+d e^{t\sqrt{c}}\right)^2-4a^2c}{\left(b^2-a^2+c+d e^{t\sqrt{c}}\right)^2-4b^2c}},\quad y=\frac{1+z}{1-z},
\end{align*}
where $a=\frac{|\alpha- \beta|}{2}, b=\frac{|\alpha+ \beta|}{2}$,
\begin{align*}
c=c(y):= K(0,y)^2+\max\{\alpha^2,\beta^2\}
\end{align*}
and
\begin{align*}
d=d(y):=-c-\alpha\beta+\frac{2c-2\sqrt{c}\sqrt{c-(c+\alpha\beta)(1-y^2)+b^2(1-y^2)^2}}{1-y^2}.
\end{align*}
\end{pro}
\begin{proof} 
In order to make easier computations, we use the M\"obius transform
\begin{align*}
z\mapsto y=\frac{1+z}{1-z} 
\end{align*}
to introduce the function $F(t,y):=H(t,z)$. Since $\frac{dy}{dz}=\frac{-2z}{(1-z)^2}$, the PDE (\ref{pde}) becomes
\begin{align}\label{pde1}
\partial_tF+\frac{y^2-1}{4}\partial_yF^2=\frac{(y^2-1)}{8y^3}\left((\alpha+\beta)^2y^4-(\alpha-\beta)^2\right).
\end{align}
 As usual, the  characteristic curve $t\mapsto (w_t(z),F(t,w_t(z)))$ associated with the PDE (\ref{pde1}) satisfies the system of ODEs:
\begin{equation}\label{Low3}
\partial_tw_{t} =\frac{1}{2}( w_{t}^2-1) F(t,w_{t}), \quad w_{0}(y) = y,
\end{equation}
\begin{equation}\label{Low4}
\partial_t \left[F(t,w_{t})\right] =\frac{(w_t^2-1)}{8w_t^3}\left((\alpha+\beta)^2w_t^4-(\alpha-\beta)^2\right),
\end{equation}
with 
\begin{align*}
w_t(y):=\frac{1+\phi_t(z)}{1-\phi_t(z)}.
\end{align*}
Combining the two last ODE's, we get
\begin{align*}
F\partial_tF=\frac{(\alpha+\beta)^2w_t^4-(\alpha-\beta)^2}{4w_t^3}.
\end{align*}
Hence, integrating with respect to $t$, we get
\begin{align*}
F(t,w_t(y))^2&=F(0,y^2)+\frac{(\alpha+\beta)^2w_t^4(y)+(\alpha-\beta)^2}{4w_t^2(y)}-\frac{(\alpha+\beta)^2y^4+(\alpha-\beta)^2}{4y^2}
\\&=1+ F^2(0,y)-F^2(\infty,y)-\frac{\alpha^2+\beta^2}{2}+\frac{(\alpha+\beta)^2w_t^4(y)+(\alpha-\beta)^2}{4w_t^2(y)}.
\end{align*}
So that, the ODE (\ref{Low3}) becomes
\begin{align*}
\partial_tw_{t}(y) =\frac{w_{t}^2(y)-1}{2} \sqrt{1+F^2(0,y)-F^2(\infty,y)-\frac{\alpha^2+\beta^2}{2}+\frac{(\alpha+\beta)^2w_t^4(y)+(\alpha-\beta)^2}{4w_t^2(y)}}.
\end{align*}
Or, equivalently
\begin{align*}
\partial_tw_{t}(y) =\frac{w_{t}^2(y)-1}{2w_{t}(y)} \sqrt{b^2w_t^4(y)+\left[1+F^2(0,y)-F^2(\infty,y)-a^2-b^2\right]w_{t}(y)^2+a^2}
\end{align*}
where $a=\frac{|\alpha- \beta|}{2}, b=\frac{|\alpha+ \beta|}{2}$.
In order to solve this last ODE, we are lead to compute the indefinite integral 
\begin{align*}
-2\int\frac{xdx}{(1-x^2)\sqrt{b^2x^4+\left(1+F(0,y)^2-F(\infty,y)^2-a^2-b^2\right)x^2+a^2}}
\end{align*}
for $y>0$.
Performing the variable change $u=1-x^2$, we transform this integral to
\begin{align*}
\int\frac{du}{u\sqrt{c-c_1u+c_2u^2}}
\end{align*}
with
\begin{align*}
c&=1+F(0,y)^2-F(\infty,y)^2,
\\c_1&=c+b^2-a^2=c+\alpha\beta,
\\c_2&=b^2=\frac{(\alpha+\beta)^2}{4}.
\end{align*}
Then writing
\begin{align*}
c&=F(0,y)^2-\frac{(\alpha+\beta)^2y^4+(\alpha-\beta)^2}{4y^2}+\frac{\alpha^2+\beta^2}{2}
\\&=F(0,y)^2-\left(\frac{|\alpha+\beta|y^2+|\alpha-\beta|}{2y}\right)^2+\frac{\alpha^2+\beta^2+|\alpha^2-\beta^2|}{2}
\\&=F(0,y)^2-\left(\frac{by^2+a}{y}\right)^2+\max\{\alpha^2,\beta^2\}
\\&=K(0,y)^2+\max\{\alpha^2,\beta^2\},
\end{align*}
we get
\begin{align*}
c_1^2-4cc_2=c^2+2c\alpha\beta+(\alpha\beta)^2-c(\alpha+\beta)^2=(c-\alpha^2)(c-\beta^2).
\end{align*}
Hence (see the proof in \cite[Theorem 3]{Dem-Hmi}), we have
\begin{align*}
\int\frac{du}{u\sqrt{c-c_1u+c_2u^2}}=\frac{1}{\sqrt{c}}\ln\frac{2c-c_1u-2\sqrt{c}\sqrt{c-c_1u+c_2u^2}}{|u|}.
\end{align*}
Let $u_t(y):=1-w_t^2(y)$, then
\begin{align*}
\frac{2c-c_1u_t(y)-2\sqrt{c}\sqrt{c-c_1u_t(y)+c_2u_t(y)^2}}{|u_t(y)|}=de^{t\sqrt{c}}
\end{align*}
for some $d=d(y,\alpha,\beta)$ and hence
\begin{align*}
2c-(c_1+\epsilon de^{t\sqrt{c}})u_t(y)=2\sqrt{c}\sqrt{c-c_1u_t(y)+c_2u_t(y)^2}
\end{align*}
where $\epsilon$ is the sign of $u$.
Raising this equality to the square and rearranging it , we get
\begin{align}\label{equality}
\left[(c_1+\epsilon de^{t\sqrt{c}})^2-4cc_2\right]u_t(y)=4c\epsilon de^{t\sqrt{c}}.
\end{align}
Equivalently,
\begin{align*}
u_t(y)=\frac{4c\tilde{d}e^{t\sqrt{c}}}{(c_1+\tilde{d}e^{t\sqrt{c}})^2-4cc_2}
\end{align*}
with $\tilde{d}=\epsilon d$. Hence
\begin{align}
w_t(y)^2&=\frac{(c_1+\tilde{d}e^{t\sqrt{c}})^2-4cc_2-4c\tilde{d}e^{t\sqrt{c}}}{(c_1+\tilde{d}e^{t\sqrt{c}})^2-4cc_2} \nonumber
\\&=\frac{(b^2-a^2+c+\tilde{d}e^{t\sqrt{c}})^2-4cb^2-4c\tilde{d}e^{t\sqrt{c}}}{(b^2-a^2+c+\tilde{d}e^{t\sqrt{c}})^2-4cb^2} \nonumber
\\&=\frac{(b^2-a^2-c+\tilde{d}e^{t\sqrt{c}})^2-4ca^2}{(b^2-a^2+c+\tilde{d}e^{t\sqrt{c}})^2-4cb^2}.\label{subordinate}
\end{align}
Finally, in order to find the value of $\tilde{d}$, we check the equality \eqref{equality} for $t=0$
\begin{align*}
\left[(c_1+\tilde{d})^2-4cc_2\right]u_0=4c\tilde{d}
\end{align*}
where $u_0:=u_0(y)=1-w_0(y)^2=1-y^2$.
  Then
  \begin{align*}
 {\tilde{d} }^2+2(c_1-\frac{2c}{u_0})\tilde{d}+c_1^2-4cc_2=0.
  \end{align*}
  The discriminant of this quadratic is
  \begin{align*}
\Delta'&=(c_1-\frac{2c}{u_0})^2-c_1^2+4cc_2
\\&=\frac{4c^2}{u_0^2}-\frac{4cc_1}{u_0}+4cc_2
\\&=\frac{4c}{u_0^2}\left(c-c_1u_0+c_2u_0^2\right)
  \end{align*}
  and hence
  \begin{align*}
 {\tilde{d} }=-c_1+\frac{2c}{u_0}\pm \frac{2\sqrt{c}}{u_0}\sqrt{c-c_1u_0+c_2u_0^2}.
  \end{align*}
  When $a=b=0$ (i.e. $c_1=c=F(0,y)^2$ and $c_2=0$), it becomes
  \begin{align*}
 {\tilde{d} }=-c+\frac{2c\pm 2cy}{1-y^2}=\frac{1\pm 2y+y^2}{1-y^2}c.
  \end{align*}
  Therefore the only solution is
  \begin{align*}
 {\tilde{d} }=-c_1+\frac{2c}{u_0}- \frac{2\sqrt{c}}{u_0}\sqrt{c-c_1u_0+c_2u_0^2}
  \end{align*}
  since for $a=b=0$, we have on the one hand by \eqref{subordinate}
  \begin{align*}
 w_t(y)=\sqrt{\frac{(-c+\tilde{d}e^{t\sqrt{c}})^2}{(c+\tilde{d}e^{t\sqrt{c}})^2}}=\frac{1-\frac{\tilde{d}}{c}e^{t\sqrt{c}}}{1+\frac{\tilde{d}}{c}e^{t\sqrt{c}}}
  \end{align*}
  on the other hand (see Remark \ref{remark}),
  \begin{align*}
w_t(y)=\frac{1+\phi_t(z)}{1-\phi_t(z)}=\frac{1+ze^{tH(0,z)}}{1-ze^{tH(0,z)}}=\frac{1+\frac{y-1}{y+1}e^{tF(0,y)}}{1-\frac{y-1}{y+1}e^{tF(0,y)}}.
\end{align*}
Hence we are done.
\end{proof}

Note that $\eta_t:\mathbb{D}\rightarrow\Omega_t$ satisfies $\eta_t(0)=0$ and 
$|\eta_t(z)|<1$ for any $z\in\mathbb{D}$. Then the characterization of the $\eta$-transform of measures on $\mathbb{T}$ in  \cite[Proposition 3.2]{Bel-Ber} implies that for any $t>0$, there exists a unique probability measure $\rho_t$ on $\mathbb{T}$ such that $\eta_t(z)=\eta_{\rho_t}(z)$ and $\phi_{t}\left(\eta_{\rho_t}(z)\right)=z$ hold for all $z\in\mathbb{D}$. 
The function $\phi_t$ satisfies the properties in \cite[Theorem 4.4, Proposition 4.5]{Bel-Ber}. Thus we have
\begin{pro}\cite[Theorem 4.4, Proposition 4.5]{Bel-Ber}\label{extension}
\begin{enumerate}
\item $\eta_t$ extends continuously to $\partial\mathbb{D}$.
\item if $\zeta\in\mathbb{T}$ satisfies $\eta_t(\zeta)\in \mathbb{D}$,  $\eta_t$ can be continued analytically to a neighborhood of $\zeta$.
\item $\Omega_{t}$ is a simply connected domain bounded by a simple closed curve.
\end{enumerate}
\end{pro}

\begin{lem}\label{distance}
The region $\overline{\Omega_t}$ does not contain 1 (resp. -1) whenever $b>0$ or $b=0$ and $\nu_0\{0\}>0$ (resp. $a>0$ or $a=0$ and $\nu_0\{\pi\}>0$).
\end{lem}
\begin{proof}

Since  $\nu_0\{0\}\geq b$ (see Proposition \ref{atom} below), by the assumption $b>0$ or $b=0$ and $\nu_0\{0\}>0$ we deduce that
\begin{align*}
\lim_{y\rightarrow+\infty}c(y)=\lim_{y\rightarrow+\infty}\left[F(0,y)-by-\frac{a}{y}\right]\left[F(0,y)+by+\frac{a}{y}\right]+\max\{\alpha^2,\beta^2\}=+\infty.
\end{align*}
Moreover, from the equality (see Proposition \ref{expression})
\begin{align*}
d(y)=c(y)\left[-1-\frac{\alpha\beta}{c(y)}+\frac{2}{1-y^2}+2\sqrt{\frac{1}{(y^2-1)^2}+\frac{c(y)+\alpha\beta}{c(y)(y^2-1)}+\frac{b^2}{c(y)}}\right],
\end{align*}
we see that, 
\begin{align*}
\lim_{y\rightarrow+\infty}d(y)=-\infty\ \ {\rm and} \quad\lim_{y\rightarrow+\infty}\frac{d(y)}{c(y)}=-1.
\end{align*}
As a result, 
\begin{align*}
w_t(y)=\sqrt{\frac{\left(b^2-a^2-c(y)+d(y) e^{t\sqrt{c(y)}}\right)^2-4a^2c(y)}{\left(b^2-a^2+c(y)+d(y) e^{t\sqrt{c(y)}}\right)^2-4b^2c(y)}}
\end{align*}
converges to 1 when $y$ goes to $+\infty$. 
Equivalently, in the $z$-variable we have, $\lim_{z\rightarrow1^-}\phi_t(z)=0$ (see Proposition \ref{expression}).
Proceeding in the same way, we prove that $\lim_{z\rightarrow-1^+}\phi_t(z)=0$. Note that, in this case, $\nu_0\{\pi\}\geq a$ and the assumption $a>0$ or $a=0$ and $\nu_0\{\pi\}>0$ implies that $\lim_{y\rightarrow 0}c(y)=+\infty$ and $\lim_{y\rightarrow 0}d(y)/c(y)=1$.
Since $\phi_t(0)=0$ and $\Omega_{t}$ is a simply connected domain bounded by a simple closed curve, we see that $\partial\Omega_{t}$ intersect $x-$axis at two points $x(t)_\pm$ from either side of the origin, with $\phi_t(x(t)_\pm)=\pm1$. From  $\lim_{z\rightarrow\pm1^\mp}\phi_t(z)=0$, we deduce that $[x(t)_-,x(t)_+]\subset(-1,1)$.
\end{proof}
\begin{cor}\label{assumption} 
For any $t>0$,
\begin{align*}
z\mapsto a\frac{1-\eta_t(z)}{1+\eta_t(z)}+b\frac{1+\eta_t(z)}{1-\eta_t(z)}
\end{align*}
is a function of Hardy class $H^{\infty}(\mathbb{D})$.
\end{cor}
\begin{proof} 
By the first item of Proposition \ref{extension}, we can easily confirm that $\eta_t$ is of hardy class $H^{\infty}(\D)$ and hence the function
\begin{align*}
z\mapsto a\frac{1-\eta_t(z)}{1+\eta_t(z)}+b\frac{1+\eta_t(z)}{1-\eta_t(z)} 
  \end{align*} 
is of hardy class $H^{\infty}(\D)$ by the previous Lemma, thanks to the fact that $\eta_t$ can not take the values $\pm1$ in $\mathbb{D}$.
\end{proof}



\section{Relationship between $\mu_t$ and $\nu_t$}

Keep the symbols $P,Q,R,S,\alpha,\beta,a,b$ and $\mu_t,\nu_t$ above. In what follows $P,Q$ and $R,S$ are associated.
Our goal here is to derive relationship between $\mu_t$ and $\nu_t$ and give more detailed properties of $\nu_t$. Here is a relationship between the corresponding sequence of moments.

\begin{pro}
For any $n \geq 1$, one has :
\begin{align}\label{Binom}
\tau[(PU_tQU_t^{*})^n]  = \frac{1}{2^{2n+1}}\binom{2n}{n}  + \frac{\tau(R+S)}{4} + \frac{1}{2^{2n}}\sum_{k=1}^n \binom{2n}{n-k}\tau((RU_tSU_t^*)^k).
\end{align}
\end{pro}
\begin{proof}
We write
\begin{equation*}
\tau[(PU_tQU_t^{*})^n] = \frac{1}{2^{2n}}\tau[(({\bf 1} + R)U_t({\bf 1} + S)U_t^{*})^n].
\end{equation*}
Let $\tilde{S}:=U_tSU_t^{*}$. Then writing
\begin{equation*}
({\bf 1} + R)U_t({\bf 1} + S)U_t^{*}=({\bf 1} + R)({\bf 1} + \tilde{S}).
\end{equation*}
one easily can see that the same enumeration techniques used in \cite[Proposition 4.1]{Dem-Ham-Hmi} to expend $\tau[(({\bf 1} + R)({\bf 1} + \tilde{S}))^n]$
 remain valid, but here we will take into account the contribution of words formed by an odd number of letters. Using the trace property and the relations $R^2=\tilde{S}^2=\bf 1$, this contribution is $\tau(R)+\tau(S)$ up to a positive integer $N$. By letting $R=S$ and using the expansion in \cite[p 1366]{Dem-Ham-Hmi}, we get $2N=2^{2n-1}$ and hence the desired equality follows.

\end{proof}

Let 
\begin{equation*}
G(t,z):= \frac{1}{z}+\sum_{n \geq 1} \frac{\tau[(PU_tQU_t^{*})^n]}{z^{n+1}}, \quad t \geq 0, |z| > 1, 
\end{equation*}
be the Cauchy transform of the process $X_t$. 
The following corollary gives a relationship between  $G$ and the Herglotz transform of $\nu_t$.
\begin{cor}
One has
\begin{equation}\label{cauherg} 
G(t,z) =\frac{1}{2z}+\frac{\alpha+\beta}{4z(z-1)}+ \frac{H(t,g(z))}{2\sqrt{z^2-z}}, \quad t \geq 0,|z|>1,
\end{equation}
where \footnote{The principal branch of the square root is taken.}
\begin{equation*}
 g(z)=2z-1+2\sqrt{z^2-z}.
\end{equation*}
\end{cor}
\begin{proof}
We will prove the following equivalent relation
\begin{equation*}
\psi_{\mu_t}(z) =\frac{(\alpha+\beta+2)z-2}{4(1-z)}+ \frac{H(t,g(1/z))}{2\sqrt{1-z}}, \quad t \geq 0,|z|<1,
\end{equation*}
 satisfied by the moment generating function of the process $X_t$ 
\begin{equation*}
\psi_{\mu_t}(z):= \sum_{n \geq 1} \tau[(PU_tQU_t^{*})^n]z^n, \quad t \geq 0, |z| < 1.
\end{equation*}

Before going into the details, recall from \cite[p. 1359]{Dem-Ham-Hmi} that $|g(1/z)|\leq |z|<1$ in the open unit disc, then this last relation
makes sense for all $|z|<1$. 
 Now multiplying (\ref{Binom}) by $z^n$ and summing  over $n\geq 1$, we get
 \begin{align*}
\psi_{\mu_t}(z)=\frac{1}{2\sqrt{1-z}}-\frac{1}{2}+\frac{(\alpha+\beta)z}{4(1-z)}+\sum_{n\geq1} \frac{z^n}{2^{2n}}\sum_{k=1}^n \binom{2n}{n-k}\tau[(RU_tSU_t^*)^k].
\end{align*}
 But, this last term rewrites, after permutation of sums and reindexing $j=n-k$, as
 \begin{align*}
\sum_{n\geq1} \frac{z^n}{2^{2n}}\sum_{k=1}^n \binom{2n}{n-k}\tau[(RU_tSU_t^*)^k]=\sum_{k\geq1} \tau[(RU_tSU_t^*)^k]\sum_{j\geq0}\frac{z^{j+k}}{2^{2j+2k}} \binom{2j+2k}{j}.
\end{align*}
Using the identity (see, e.g. \cite{Dem-Ham-Hmi})
\begin{equation*}
 \sum_{j\geq0}\binom{2j+2k}{j}\frac{z^j}{2^{2j}}=\frac{2^{2k}}{\sqrt{1-z}}(1+\sqrt{1-z})^{-2k},\quad |z|<1,
\end{equation*}
we get
\begin{align*}
\psi_{\mu_t}(z) &=\frac{1}{2\sqrt{1-z}}-\frac{1}{2}+\frac{(\alpha+\beta)z}{4(1-z)}+\frac{1}{\sqrt{1-z}}\sum_{k\geq1}\frac{\tau[(RU_tSU_t^*)^k]z^k}{(1+\sqrt{1-z})^{2k}}
\\&=\frac{1}{2\sqrt{1-z}}-\frac{1}{2}+\frac{(\alpha+\beta)z}{4(1-z)}+\frac{1}{\sqrt{1-z}}\frac{H(t,g(1/z))-1}{2}
\\&=-\frac{1}{2}+\frac{(\alpha+\beta)z}{4(1-z)}+\frac{H(t,g(1/z))}{2\sqrt{1-z}},
\end{align*}
which proves the corollary.
\end{proof}

We are now ready to prove the relationship between the spectral measure of $X_t$ and $Y_t$: $\mu_t \leftrightsquigarrow \nu_t$.
\begin{teo}
Let $\tilde{\mu}_t(d\theta)$ be the positive measure on $[0,\pi]$ obtained from $\mu_t(dx)$ via the variable change $x=\cos^2(\theta/2)$ and  $\hat{\mu}_t:=\frac{1}{2}\left(\tilde{\mu}_t+\left(\tilde{\mu}_t|_{(0,\pi)}\right)\circ j^{-1}\right)$ it's symmetrization on $(-\pi,\pi)$ with the mapping $j:\theta\in(0,\pi)\mapsto-\theta\in(-\pi,0)$. Then, the two measures $\mu_t$ and $\nu_t$ are related via 
\begin{equation}\label{measure}
\nu_t=2\hat{\mu}_t-\frac{2-\alpha-\beta}{2}\delta_{\pi}-\frac{\alpha+\beta}{2}\delta_0.
\end{equation}
\end{teo}
\begin{proof}
By \eqref{cauherg}, we have
\begin{align*}
H(t,g(z)) &=2\sqrt{z^2-z}\left(G(t,z) -\frac{2-\alpha-\beta}{4z}-\frac{\alpha+\beta}{4(z-1)} \right).
\end{align*}
Letting $\tilde{\mu}_t(d\theta)=\mu_t(dx)$ with $x=\cos^2(\theta/2), \theta\in[0,\pi]$, we get
\begin{align*}
H(t,g(z)) &=-2\sqrt{z^2-z}\left(\int_0^{\pi}\frac{1}{z-\cos^2\frac{\theta}{2}}\tilde{\mu}_t(d\theta) -\frac{2-\alpha-\beta}{4z}-\frac{\alpha+\beta}{4(z-1)} \right).
\end{align*}
Next, we perform the variable change
\begin{equation*}
\zeta:=g(z)=2z-1+2\sqrt{z^2-z} \Leftrightarrow z=\frac{2+\zeta+\zeta^{-1}}{4},
\end{equation*}
to get
\begin{align*}
H(t,\zeta) &=\frac{\zeta^{-1}-\zeta}{2}\left(\int_0^{\pi}\frac{1}{\frac{2+\zeta+\zeta^{-1}}{4}-\cos^2\frac{\theta}{2}}\tilde{\mu}_t(d\theta) -\frac{2-\alpha-\beta}{2+\zeta+\zeta^{-1}}-\frac{\alpha+\beta}{-2+\zeta+\zeta^{-1}} \right)
\\&=\int_0^{\pi}\frac{2(\zeta^{-1}-\zeta)}{2+\zeta+\zeta^{-1}-4\cos^2\frac{\theta}{2}}\tilde{\mu}_t(d\theta) -\frac{(2-\alpha-\beta)(1-\zeta)}{2(1+\zeta)}-\frac{(\alpha+\beta)(1+\zeta)}{2(1-\zeta)}.
\end{align*}
But since
\begin{align*}
 \frac{\zeta^{-1}-\zeta}{2+\zeta+\zeta^{-1}-4\cos^2\frac{\theta}{2}}&=\frac{\zeta^{-1}-\zeta}{\zeta+\zeta^{-1}-2\cos\theta}
 \\&=\frac{1-\zeta^2}{\zeta^2-2\zeta\cos\theta+1}
 \\&=\frac{e^{i\theta}}{e^{i\theta}-\zeta}+\frac{e^{i\theta}}{e^{-i\theta}-\zeta}-1,
\end{align*}
then
\begin{align*}
H(t,\zeta) &=2\int_0^{\pi}\left( \frac{e^{i\theta}}{e^{i\theta}-\zeta}+\frac{e^{-i\theta}}{e^{-i\theta}-\zeta}-1\right)\tilde{\mu}_t(d\theta) -\frac{(2-\alpha-\beta)(1-\zeta)}{2(1+\zeta)}-\frac{(\alpha+\beta)(1+\zeta)}{2(1-\zeta)}.
\end{align*}
Thus, using the symmetrization $\hat{\mu}_t:=\frac{1}{2}\left(\tilde{\mu}_t+\left(\tilde{\mu}_t|_{(0,\pi)}\right)\circ j^{-1}\right)$ with $j:\theta\in(0,\pi)\mapsto-\theta\in(-\pi,0)$, we get
\begin{align*}
H(t,\zeta) &=\int_{-\pi}^{\pi}\frac{e^{i\theta}+\zeta}{e^{i\theta}-\zeta}(2\hat{\mu}_t -\frac{2-\alpha-\beta}{2}\delta_{\pi}-\frac{\alpha+\beta}{2}\delta_0)(d\theta).
\end{align*}
This proves the theorem.
\end{proof}
\begin{rem}\label{relationship}
The relationship $\mu_t\leftrightsquigarrow\nu_t$ enable us, in particular, to 
retrieve the decomposition of $\nu_\infty$ already obtained in section 2 from the spectral measure $\mu_\infty$ (given by the free multiplicative convolution of the spectral measure of $P$ and $UQU^*$ with $U$ is a Haar unitary free from $\{P,Q\}$ (see, \cite[Example 3.6.7]{Dyk-Nic-Voi})).  Indeed, we have $\hat{\delta_0}=\delta_\pi,\hat{\delta_1}=\delta_0$ and if $\mu_t$ has the density $h(x)$ with respect to $dx$ on $[0,1]$, then  $\nu_t$ has the density $\hat{h}(\theta)$ with respect to the (no-normalized) Lebesgue measure $d\theta$ on $\mathbb{T}=(-\pi,\pi]$ with $\hat{h}(\theta)=h(\cos^2(\theta/2))|\sin\theta|/4$. 
\end{rem}

By virtue of  the fact that $P$ and $U_tQU_t^*$ are in generic position for any $t>0$ (see, e.g., \cite[Remark 3.5]{Izu-Ued}), we have
\begin{pro}\label{atom}
For every $t>0$,
the positive measure 
$
\sigma_t:=\nu_t-a\delta_{\pi}-b\delta_0
$
has no atom at both 0 and $\pi$.
Moreover, at $t=0$, we have $\sigma_0\{0\}\geq0$ and $\sigma_0\{\pi\}\geq0$ with equalities (i.e. $\sigma_0$ has no atom at both 0 and $\pi$), if and only if the projections $P$ and $Q$ are in generic position. 
\end{pro}
\begin{proof}
By \eqref{measure}, we have
\begin{align*}
\sigma_t&=2\hat{\mu_t}-\frac{2-\alpha-\beta+|\alpha-\beta|}{2}\delta_{\pi}-\frac{\alpha+\beta+|\alpha+\beta|}{2}\delta_0
\\&=2\hat{\mu_t}-(1-\min\{\alpha,\beta\})\delta_{\pi}-\max\{\alpha+\beta,0\}\delta_0.
\end{align*}
Since $\alpha=2\tau(P)-1$ and $\beta=2\tau(Q)-1$, 
\begin{align}\label{sigma}
\sigma_t&=2\left[\hat{\mu_t}-(1-\min\{\tau(P),\tau(Q)\})\delta_{\pi}-\max\{\tau(P)+\tau(Q)-1,0\}\delta_0\right].
\end{align}
The desired assertion immediately follows from \cite[Proposition 3.1]{Izu-Ued}.
\end{proof}

\begin{pro}\label{continspectrum}
For every $t>0$, 0 and $\pi$ does not belong to the continuous singular spectrum of $\sigma_t$.
\end{pro}
\begin{proof}
Let
\begin{align*}
L(t,z):=\int_{\mathbb{T}}\frac{e^{i\theta}+z}{e^{i\theta}-z}d\sigma_t(\theta)=H(t,z)-a\frac{1-z}{1+z}-b\frac{1+z}{1-z}.
\end{align*}
Then, \eqref{defK} rewrites as
\begin{align*}
K(t,z)^2&=H(t,z)^2-\left(a\frac{1-z}{1+z}+b\frac{1+z}{1-z} \right)^2
\\&=L(t,z)\left(L(t,z)+2a\frac{1-z}{1+z}+2b\frac{1+z}{1-z}\right).
\end{align*}
But from the second item in Proposition \ref{extension} together with the subordination relation in Proposition \ref{subor}, $K(t,.)$ has an analytic continuation in some neighborhoods of $\pm1$. Moreover,
\begin{align*}
\lim_{z\rightarrow\pm1^\mp}K(t,z)=\lim_{z\rightarrow\pm1^\mp}K(0,\eta_t(z))=K(0,x(t)_\pm)
\end{align*}
where $x(t)_\pm$ are the real boundaries of $\Omega_t$ (see the proof of Lemma \ref{distance}). Thus,
\begin{align*}
K(0,x(t)_\pm)^2=\lim_{z\rightarrow\pm1^\mp}L(t,z)\left(L(t,z)+2a\frac{1-z}{1+z}+2b\frac{1+z}{1-z}\right).
\end{align*}
Since
\begin{align*}
L(t,z)+2a\frac{1-z}{1+z}+2b\frac{1+z}{1-z}
\end{align*}
blows up as $z\rightarrow\pm1^\mp$,
$\lim_{z\rightarrow\pm1^\mp}L(t,z)=0$.
Consequently, the Poisson transform of $\sigma_t$, which is nothing but the real part of $L(t,z)$, vanishes as  $z\rightarrow\pm1^\mp$ and hence the desired assertion follows from Proposition 1.3.11 and equation (1.8.8) in \cite{CMR}.
\end{proof}
\begin{rem}
Note that when $\alpha=\beta=0$ (i.e. $\tau(P)=\tau(Q)=1/2$), the two measures $\sigma_{t/2}$ and $\gamma_{t/2}$ (recall the definition of $\gamma_t$ from \eqref{gamma}) coincide  with the spectral measure of the product of the free unitary Brownian motion with a free unitary operator whose distribution is $\sigma_0=2\hat{\mu_0}$. 
\end{rem}

\section{Free mutual information and orbital free entropy}
Here is our main application to the proof of the conjecture $i^*=-\chi_{orb}$. For a pair of projections $(P,Q)$,
we use the same definitions of the free mutual information $i^*( \mathbb{C}P+\mathbb{C}(I-P); \mathbb{C}Q+\mathbb{C}(I-Q) )$ (hereafter $i^*(P:Q)$) and the orbital free entropy $\chi_{orb}(P,Q)$ as expounded in the last section of the paper \cite{Izu-Ued}. We rerfer the reader to \cite{Hia-Pet, Hia-Ued, Voiculescu} for more information. Using subordination technology,  a partial result for the identity $i^*(P:Q)=-\chi_{orb}(P,Q)$ is obtained in \cite[Lemma 4.4]{Izu-Ued} (note that the function $H$ there is exactly  our $\frac{1}{4}K^2$). The result is as follows.
\begin{lem}(\cite{Izu-Ued}).\label{izue} 
If $K(t,.)$ define a function of Hardy class $H^{3}(\mathbb{D})$ for any $t>0$, then 
$
i^*\left( \mathbb{C}P+\mathbb{C}(I-P); \mathbb{C}Q+\mathbb{C}(I-Q) \right)=-\chi_{orb}\left(P,Q\right)
$.
\end{lem}
Let $L(t,z)$ be as in the proof of Proposition \ref{continspectrum}. 
From
\begin{align*}
(\Re L(t,z))^2\leq\Re L(t,z)\Re\left(L(t,z)+2a\frac{1-z}{1+z}+2b\frac{1+z}{1-z}\right)\leq |K(t,z)^2|,
\end{align*}
the assumption in Lemma \ref{izue} implies that $\sigma_t$ has an $L^3$-density. The converse remains true; i.e. if $\sigma_t$ has an $L^3$-density for any $t>0$, then $K(t,.)$ becomes a function of Hardy class $H^{3}(\mathbb{D})$. In fact, according to \cite[Theorem 1.7, p.208]{Con}, $L(t,.)$ is a function of Hardy class $H^{3}(\mathbb{D})$. On the other hand, from Propositions \ref{atom} and \ref{continspectrum}, we see that $L(t,.)$ has an analytic continuation across both points $\pm 1$. Moreover, the limit $\lim_{z\rightarrow\pm1}L(t,z)=0$ implies that the constant term in the power series expansion around $z=\pm1$ is zero. So that $L(t,z)\left(a\frac{1-z}{1+z}+b\frac{1+z}{1-z}\right)$ is bounded in some neighborhoods at both $\pm1$. Hence 
\begin{align*}
K(t,z)^2=L(t,z)^2+2L(t,z)\left(a\frac{1-z}{1+z}+b\frac{1+z}{1-z}\right)
\end{align*}
becomes a function  of Hardy class $H^{3/2}(\mathbb{D})$. From this discussions, we deduce that
\begin{lem}\label{assump1}
If $\sigma_t$ has an $L^3$-density for every $t>0$, then 
$
i^*\left(P:Q \right)=-\chi_{orb}\left(P,Q\right)
$.
\end{lem}
Here we reprove the same result by an equivalent but more handy assumption.
\begin{pro}\label{entropy}
Assume that for every $t>0$, $H(0,\eta_t(.))$ is a function of Hardy class $H^{3}(\mathbb{D})$. Then the equality
$
i^*\left( P:Q \right)=-\chi_{orb}\left(P,Q\right)
$
holds.
\end{pro}
\begin{proof}
We will prove that the assumptions $H(0,\eta_t(.))\in H^{3}(\mathbb{D})$ and $K(t,.)\in H^{3}(\mathbb{D})$ are equivalent and so we can use the result of Lemma \ref{izue}. To this end, we use the subordination relation in Proposition \ref{subor} together with \eqref{defK}, to write
\begin{align*}
K(t,z)^2=H(0,\eta_t(z))^2-\left(a\frac{1-\eta_t(z)}{1+\eta_t(z)}+b\frac{1+\eta_t(z)}{1-\eta_t(z)}\right)^2.
  \end{align*}
But, the function (see Corollary \ref{assumption})
\begin{align*}
z\mapsto \left(a\frac{1-\eta_t(z)}{1+\eta_t(z)}+b\frac{1+\eta_t(z)}{1-\eta_t(z)} \right)^2
  \end{align*} 
is of hardy class $H^{\infty}(\D)$. Hence we are done.
\end{proof}
The benefit of the above assumption  is that it transfers the necessary regularity of $\sigma_t$ and hence of $\nu_t$ for $t>0$ to an equivalent regularity for $\nu_0$ in connection with the conformal transformation $\eta_t$. 
Immediately from this assumption, we can see that the equality $i^*(P:Q)=-\chi_{orb}(P,Q)$ holds when the two initial operators $P,Q$ are assumed to be classically or freely independent. In fact, if $P,Q$ are classically independent, then $R,S$ become two independent symmetries, so that
\begin{equation*}
\tau[(RS)^n]=\tau(R^n)\tau(S^n)=f_n(0)=
\begin{cases}
1,\quad n \ {\rm  even}\\
\alpha\beta \quad  n \ {\rm odd}
\end{cases}.
 \end{equation*} 
Hence, we can compute explicitly the initial data
 \begin{equation*}
 H(0,z)=1+2\sum_{n\geq 1} f_n(0) z^n=\frac{1+2\alpha\beta z+z^2}{1-z^2}.
\end{equation*} 
Whereas, when $P$ and $Q$ are freely independent, we have from Proposition \ref{freeconv}
\begin{align*}
H(0,z)=\sqrt{1+4 z\frac{\alpha \beta \left(1+z\right)^2+ \left(\alpha-\beta \right)^2z}{\left(1-z^2\right)^2}}
\end{align*}
and in both cases, we see  that $H(0,\eta_t(.))\in H^{\infty}(\mathbb{D})$.
Here is a sample application of Proposition \ref{entropy}  improving the result in \cite[Corollary 4.5]{Izu-Ued}.
\begin{lem}\label{assump2}
Assume that $\sigma_0$ has an $L^{3}$-density with respect to $d\theta$. Then  
$
i^*\left(P:Q \right)=-\chi_{orb}\left(P,Q\right)
$.
\end{lem}
\begin{proof} 
Under the assumption here and according to \cite[Theorem 1.7, p.208]{Con}, $L(0,z)$ is a function of Hardy class $H^{3}(\mathbb{D})$ and hence so does  $L(0,\eta_t(z))$ too by Littlewood's subordination theorem (see \cite[Theorem 1.7]{Dur}). 
 On the other hand, by Corollary \ref{assumption}, the function
\begin{align*}
z\mapsto a\frac{1-\eta_t(z)}{1+\eta_t(z)}+b\frac{1+\eta_t(z)}{1-\eta_t(z)} 
  \end{align*} 
is of hardy class $H^{\infty}(\D)$. Hence, 
\begin{align*}
H(0,\eta_t(z))=L(0,\eta_t(z))+a\frac{1-\eta_t(z)}{1+\eta_t(z)}+b\frac{1+\eta_t(z)}{1-\eta_t(z)}
  \end{align*}  
is of hardy class $H^{3}(\D)$ and then  we are done thanks to Proposition \ref{entropy}.
\end{proof}

We can now prove the main result of this section.

\begin{teo}
For any two projections $P,Q$, if
\begin{align*}
\mu_t-(1-\min\{\tau(P),\tau(Q)\})\delta_{0}-\max\{\tau(P)+\tau(Q)-1,0\}\delta_1
\end{align*}
has an $L^3$-density with respect to $x(1-x)dx$ on $[0,1]$ for $t=0$ or every $t>0$, then $i^*\left(P:Q \right)=-\chi_{orb}\left(P,Q\right)$.
\end{teo}
\begin{proof} 
From  the relationship $\mu_t\leftrightsquigarrow\nu_t$ (together with Remark \ref{relationship}), the assumption here implies that
\begin{align*}
\hat{\mu_t}-(1-\min\{\tau(P),\tau(Q)\})\delta_{\pi}-\max\{\tau(P)+\tau(Q)-1,0\}\delta_0
\end{align*}
  has an $L^3$-density with respect to $d\theta$ on $\mathbb{T}=(-\pi,\pi]$ for $t=0$ or every $t>0$ and hence by \eqref{sigma}, the measure $\sigma_t$ does so also. The desired identity immediately follows from Lemma \ref{assump1} and Lemma \ref{assump2}.
\end{proof}


\end{document}